\documentclass[final,10pt]{elsart}

\usepackage{amssymb,latexsym,amsmath,amsfonts}
  \usepackage[dvips]{graphicx}
\usepackage{subfigure}
\usepackage{hyperref}

\usepackage{pdfpages}
\usepackage{graphicx}
\usepackage{pgfplots}
\usepackage{booktabs}
\usepackage{color}
\usepackage{caption}
\usepackage{moreverb}
\usepackage{cmbright}
\usepackage[utf8]{inputenc}

\newcommand{\be}{\begin{equation}}
\newcommand{\en}{\end{equation}}
\newcommand{\bea}{\begin{eqnarray}}
\newcommand{\ena}{\end{eqnarray}}
\newcommand{\beano}{\begin{eqnarray*}}
\newcommand{\enano}{\end{eqnarray*}}
\newcommand{\bee}{\begin{enumerate}}
\newcommand{\ene}{\end{enumerate}}

\newcommand{\Hil}{{\cal H}}

\newcommand{\F}{{\cal F}}
\newcommand{\J}{{\cal J}}

\newcommand{\E}{{\cal E}}
\newcommand{\lmin}{{ l}}

\newcommand{\1}{1 \!\! 1}

\newenvironment{proof}{\noindent {\bf Proof:}}{\hfill$\Box$}

\begin{document}
\begin{frontmatter}

\title{Projector operators in clustering}

\author[autore,autore2]{F.Bagarello}
\ead{fabio.bagarello@unipa.it}
\author[autore]{M. Cinà}
\ead{marcocina89@gmail.com}
\author[autore]{F. Gargano}
\ead{francesco.gargano@unipa.it}
\journal{: Mathematical Methods in the Applied Sciences}
\maketitle
\date{}
\address[autore]{Dept. of Energy, Engineering of Information and Mathematical Models, University of Palermo, Italy.}

\address[autore2]{I.N.F.N, Torino, Italy.}

\begin{abstract}
\noindent In a recent paper the notion of {\em quantum perceptron} has been introduced in connection with projection operators. Here we extend this idea, using these kind of operators to produce a {\em clustering machine}, i.e. a  framework which generates different clusters from a set of input data. Also, we consider what happens when the orthonormal bases first used in the definition of the projectors are replaced by frames, and how these can be useful when trying to connect  some noised signal to a given cluster.
\end{abstract}

\end{frontmatter}

\vspace{-6pt}
\section{Introduction}\label{sec:intro}

As it is well known, a single perceptron cannot be used to model a XOR, since the possible outputs are not separated by one line in the
${\Bbb R}^2$ plane,  \cite{MM69}. Of course, this difficulty can be overcame easily by using more perceptrons together, that is, by using a (very simple) artificial neural network.
In a recent paper, \cite{quper}, the author suggested a possible solution which he  called {\em quantum perceptron},
mainly because he uses tools which are quite common in quantum mechanics, and in particular orthogonal bases, projector operators,
resolutions of the identity, and so on. However, since nothing really {\em quantum} appears in his treatment, (no Schr\"odinger dynamics, no Heisenberg picture, no uncertainty relation,...), we will not follow his terminology here. Nevertheless, following the main ideas proposed in \cite{quper}, in this paper we try to propose a systematic way to use tools arising
in functional analysis and operator theory, and used in the mathematical framework of quantum mechanics, for clustering problems,
like the one of the XOR. We refer to \cite{mes} and \cite{rs} as basic textbooks on quantum mechanics
and functional analysis respectively. We will see that these tools give rise to several possible procedures which can be of some utility when trying to identify
or collect objects (i.e. signals), that are vectors in some Hilbert space.
In particular, we shall construct some projector operators which allows to introduce an extended concept of distance that
can be used also in clustering procedure, i.e. in procedures where a generic input signal need to be
classified according to some \textit{similarity} properties characterizing  the signals we are analyzing. We shall apply our theoretical framework
to concrete cases like the XOR, the clustering of colors,  the classification of the tones played by an instrument and the diagnosis of the celiac disease.
All these cases can be essentially connected with the concept of \textit{clustering} and \textit{classification of patterns} (see \cite{Jai} for
an overview of pattern clustering methods). Moreover, we shall also discuss in this paper how the clustering procedure proposed can be improved
with the aid of the \textit{finite unit norm tight frames} (\cite{GKK01,BF03,CK03, SIS13,ole,dau,heil}) if we introduce some disturbances in the signals we are analyzing.

The paper is organized as follows: in the next section we introduce the general idea of how our clustering machine (CM) should work,
and we propose some applications; in particular in Section \ref{sec:gates} we consider the  XOR and the OR gates, in Section \ref{sec:rgb}
we apply our CM to the clustering of colors in the RGB framework,  in Section \ref{sec:tones} we deal with the recognition of tones played by some instruments, and in Section \ref{sec:celiac} we apply our methodology to the diagnosis of the celiac disease.
In Section \ref{sec:frames} we show how to extend the same settings by replacing orthonormal (o.n.) bases with frames,
and how this opens the possibility of correctly recognizing the cluster of some given noised signal.
Some explicit applications are also considered. Section \ref{sec:conclusion} contains our conclusions, and in the Appendix we introduce some properties of the frames.

\section{Stating the problem and first considerations}\label{sec:framework}

\vspace{3mm}

Suppose we have $N$ possible inputs $I_j$ which correspond to $M\leq N$ different outputs $O_\alpha$.  We would like to construct a CM which is able, performing a single simple operation, to tell us which are the possible inputs that have produced an observed result. We consider a Hilbert space $\Hil$ made of all the vectors representing all the possible inputs from which we can obtain the outputs $O_{\alpha}$.
We introduce $N$ orthogonal and normalized vectors which we write as $e_{\alpha,k}$, where $\alpha=1,2,\ldots,M$ and $k=1,2,\ldots,N_\alpha$, $N_\alpha$
being the different inputs $I_j$'s giving rise to the same output $O_\alpha$. Of course, we have  $\sum_{\alpha=1}^M N_\alpha=N$.
We assume that the inputs $I_j$, a priori, do not cover all the possibilities giving raise to $O_\alpha$ and hence the set $\E=\{e_{\alpha,k}\}$ could not be an o.n. basis for $\Hil$. What we know for sure is that $\E$ is an o.n. set. For each fixed $\alpha$, let $\Hil_\alpha$ be the finite linear span of the $e_{\alpha,k}$'s, $k=1,2,\ldots,N_\alpha$. Then $\dim(\Hil_\alpha)=N_\alpha$. Of course $\E$ is complete in $\Hil$ if and only if $\oplus_{\alpha=1}^M \Hil_\alpha=\Hil$. In general, however, we can only say that $\Hil\supseteq \oplus_{\alpha=1}^M \Hil_\alpha$. Notice that every $f\in\Hil_\alpha$ is necessarily orthogonal to each $g\in\Hil_\beta$, if $\alpha\neq\beta$: $\left<f,g\right>=0$.

Let us now introduce first the $N$ orthogonal projection operators $P_{\alpha,k}=|e_{\alpha,k}\left>\right<e_{\alpha,k}|$ satisfying $P_{\alpha,k}P_{\beta,n}=\delta_{\alpha,\beta}\delta_{k,n}P_{\alpha,k}$ and $\sum_{\alpha,k}\|P_{\alpha,k}f\|^2\leq \|f\|^2$, for all $f\in\Hil$. Of course if $\E$ is complete in $\Hil$, then the strict equality holds.
Now, out of the $P_{\alpha,k}$'s, $M$ orthogonal projection operators can be defined, one for each output $O_\alpha$: $Q_\alpha=\sum_{k=1}^{N_\alpha}P_{\alpha,k}$. They satisfy $Q_\alpha Q_\beta=\delta_{\alpha,\beta}Q_\alpha$, $\alpha, \beta=1,2,\ldots,M$, and, if $\E$ is complete in $\Hil$, then $\sum_{\alpha=1}^M Q_\alpha=\1$. $\{Q_1,\ldots,Q_M\}$ is our CM, which we are going to use as we will explain in detail.
Let $f\in\Hil$ be the signal we want to analyze, and let us construct the related set $\{q_1(f),\ldots,q_M(f)\}$, where
$$q_\alpha(f)=\|Q_\alpha f\|^2=\sum_{k=1}^{N_\alpha}\|P_{\alpha,k}f\|^2=\sum_{k=1}^{N_\alpha}|\left<e_{\alpha,k},f\right>|^2.$$
We obviously have  $\sum_{\alpha=1}^M q_\alpha(f)\leq\|f\|^2$ for all $f\in\Hil$, and the equality holds only if $\E$ is complete in $\Hil$.

To explain how the CM works concretely, we now consider what happens in several different cases. Also, we will assume here that the signal we want to analyze is normalized, $\|f\|=1$. However, it is worth noticing that this is not always possible in concrete applications, as in the RGB example we will consider later, since, in that particular case, normalizing the original signal produces a change in the colors.

\begin{enumerate}

\item We first consider the case in which $q_{\alpha_0}(f)=1$, for some $\alpha_0\in \{1,2,\ldots,M\}$. This implies that $q_{\beta}(f)=0$ for all $\beta\neq\alpha_0$. Hence the interpretation is clear: $f$ corresponds to the output $O_{\alpha_0}$, and therefore it is necessarily some linear combination of the vectors of the o.n. basis of $\Hil_{\alpha_0}$, $\{e_{{\alpha_0},k},\,k=1,2,\ldots,N_{\alpha_0}\}$.

\item Suppose now that, for some $\alpha_0\in \{1,2,\ldots,M\}$, $q_{\alpha_0}(f)\simeq1$. Then,  $q_{\beta}(f)\simeq0$ for all $\beta\neq\alpha_0$. Again the interpretation is clear, but has a certain alea: we are {\em almost sure} that $f$ corresponds to the output $O_{\alpha_0}$, even if there exists a small probability that this is not the case. Of course, the closer $q_{\alpha_0}(f)$ is to one, the smaller the probability that the output  is not $O_{\alpha_0}$.

\item Suppose now that there exist more than one index such that $q_\alpha(f)\neq0$. To simplify the situation, let us suppose that $q_1(f)$ and $q_2(f)$ are both non zero, and that $q_3(f)=\cdots=q_M(f)=0$. Then we may have two different situations:
    \begin{enumerate}
    \item the first case is when $q_1(f)+q_2(f)\simeq1$. In this case, we can essentially exclude all the outputs except the first two,
    $O_1$ and $O_2$, and which one between these two is the most probable depends on the difference between $q_1(f)$ and $q_2(f)$: for instance, if $q_1(f)>q_2(f)$, then $f$, corresponds to $O_{1}$ more probably than to $O_{2}$. For sure, it does not correspond to the other $O_\alpha$. From this point of view, a sort of {\em degree of membership}, \cite{gog}, related to the values of the $q_\alpha(t)$'s, could be introduced in our treatment.
    \item
     {the second case is when  $q_1(f)+q_2(f)$ is significantly smaller than 1.
    This clearly implies also that $\sum_{\alpha=1}^M q_\alpha(f)$ is less than one.
    But, since $\|f\|=1$, the conclusion is that a new output $O_{M+1}$ is missing in our original set, and must be taken into account. Then $\E$ cannot be complete in $\Hil$:
    there is at least one unit vector, $e_{M+1,1}$, which is orthogonal to all the other vectors in $\E$.
    Hence $\dim(\Hil)\geq N+1$. The vector $e_{M+1,1}$ can be constructed using a sort of Gram-Schmidt
    orthogonalization procedure: we put first $\tilde f:=f-Q_1(f)-Q_2(f)$, and then we define $e_{M+1,1}=\frac{\tilde f}{\|\tilde f\|}$.
    One can check that $e_{M+1,1}$ has all the required properties: it is normalized, and is orthogonal to all the
    vectors in $\E$. The related projection operator is $Q_{M+1}=P_{M+1,1}=|e_{M+1,1}\left>\right<e_{M+1,1}|$. Now, by construction, we have that
    $q_1(f)+q_2(f)+q_{M+1}(f)=1$, where $q_{M+1}(f)$ is defined in
    analogy with the others: $q_{M+1}(f)=\|Q_{M+1}f\|^2=|\left<e_{M+1,1},f\right>|^2$. It is important to stress that, as we will discuss again,
    we cannot be sure that the set $\E\cup\{e_{M+1,1}\}$ is complete in $\Hil$. In other words, it could further happen that $\dim(\Hil)>N+1$, and this is clearly the case if we find
    a new signal $g$ such that $q_{\alpha}(g)\approx 0 \quad \forall \alpha=1,\ldots,M+1$ .}
    \end{enumerate}
\item
 {Suppose finally that  $q_\alpha(f)=0$ for all $\alpha=1,2,\ldots,M$. Then we are back to a situation similar to the one just considered:
it surely exists a new vector, $e_{M+1,1}$, which is such that $q_{M+1,1}(f)=1$. Actually, since $f$ is already normalized and it is orthogonal
to all the vectors in $\E$, it is enough to put $e_{M+1,1}=f$. As before, we conclude that $\dim(\Hil)$ is, at least, equal to $N+1$. However, it might happen that for
some other signal, $g\in\Hil$, again with $\|g\|=1$, we get again $q_j(g)=0$, for all $j=1,2,\ldots,M+1$. Then we are forced to conclude
that $\E\cup\{e_{M+1,1}\}$ is not yet complete in $\Hil$, and a second vector $e_{M+2,1}=g$ must be added to this set.
Of course, this might happen several times. However, if for all the signals $f_{rel}$ which are {\em relevant} for us,
it happens that $\sum_{\alpha=1}^{M+2}q_\alpha(f_{rel})\simeq 1$, we can conclude that, at least for our purposes,
the {\em effective} dimension of $\Hil$ is exactly $N+2$.} It is not difficult now generalize further these results.
\end{enumerate}

An apparently different way to compare signals using tools coming from functional analysis is based on the following idea: confider two signals $f_1, f_2\in\Hil$. Then, the Schwarz inequality implies that $\left|\left<f_1,f_2\right>\right|\leq\|f_1\|\|f_2\|$.
Let us now define the following non negative function on $\Hil\times\Hil$:
\be
F[f_1,f_2]:=\|f_1\|\|f_2\|-\left|<f_1,f_2>\right|.
\label{funF}\en
It is clear that $F[f_1,f_2]\geq0$ for all $f_1, f_2\in\Hil$. Morevover: (a) if $f_1=f_2$, then $F[f_1,f_1]=0$; (b) if $f_1$ is orthogonal to $f_2$, then $F[f_1,f_2]=\|f_1\|\|f_2\|$. In particular $F[f_1,f_2]=1$ if they are both normalized.
Finally, (c) suppose that $f_1\simeq f_2$. This means for us that $q_\alpha(f_1-f_2)$ is sufficiently small, $\forall \alpha$
or, which is the same if $\E$ is an o.n. basis for $\Hil$, that $\|f_1-f_2\|$ is sufficiently small.
Therefore, since $0\simeq \|f_1-f_2\|^2=\|f_1\|^2+\|f_2\|^2-2\Re\left<f_1,f_2\right>$, we deduce that
$2\Re\left<f_1,f_2\right>\simeq \|f_1\|^2+\|f_2\|^2$. Now, assuming for simplicity that $\left<f_1,f_2\right>\in {\Bbb R}_+$,
which is always the case in our explicit applications, we deduce that
$$
F[f_1,f_2]\simeq\frac{1}{2}\left|(\|f_1\|-\|f_2\|)\right|^2,
$$

which is clearly expected to be small when $f_1\simeq f_2$.

Summarizing, when $f_1=f_2$ then $F[f_1,f_2]=0$. When $f_1$ and $f_2$ are {\em essentially different}, i.e. when they are orthogonal, then $F[f_1,f_2]$ is large. Finally, if $f_1$ is close to $f_2$, then $F[f_1,f_2]$ is close to zero. It is important to stress that $F[f_1,f_2]$ must be used cum grano salis: in fact, if $f_2=\alpha f_1$, for some $\alpha\in\Bbb C$, again we deduce that $F[f_1,f_2]=0$. Hence our previous statement cannot be inverted: if $F[f_1,f_2]=0$ this does not imply that $f_1=f_2$! However, this cannot happen if we, for instance, restrict to those normalized signals satisfying $\left<f_1,f_2\right>\in {\Bbb R}_+$. In this case $F[f_1,f_2]=0$ if and only if $f_1=f_2$.

\subsection{A first simple example: the XOR and the OR gates}\label{sec:gates}

We discuss first the easiest examples, i.e. the XOR gate and then the OR gate. This is useful to clarify the ideas. Later on the same ideas will be applied to more complicated examples, living in high-dimensional Hilbert spaces.

Following our procedure, since we know that there are exactly 4 inputs and 2 outputs, we associate each input $I_j$ with the vector $e_j$ of the canonical o.n. basis of $\Hil={\Bbb R}^4$. We recall that $e_j$ has three zero entries, while the $j$-th entry is equal to one. Then, $P_j=|e_j\left>\right<e_j|$ is a $4\times4$ matrix with all zero elements except the $j$-th one in the main diagonal, which is equal to one. Then $Q_1=P_1+P_4=diag(1,0,0,1)$ and $Q_2=P_2+P_3=diag(0,1,1,0)$. A generic input $f\in\Hil$ has the form
$$
f=\left(
    \begin{array}{c}
      f_1 \\
      f_2 \\
      f_3 \\
      f_4 \\
    \end{array}
  \right),
$$
with $\sum_{j=1}^4|f_j|^2=1$. Then, since
$$
Q_1f=\left(
    \begin{array}{c}
      f_1 \\
      0 \\
      0 \\
      f_4 \\
    \end{array}
  \right), \qquad Q_2f=\left(
    \begin{array}{c}
      0 \\
      f_2 \\
      f_3 \\
      0 \\
    \end{array}
  \right),
$$
we get $q_1(f)=\|Q_1f\|^2=|f_1|^2+|f_4|^2$ and $q_2(f)=\|Q_2f\|^2=|f_2|^2+|f_3|^2$. Suppose the signal $f$ coincides with one of the vectors of $\E$, for instance with $e_1$. Then, it is clear that $\{q_1(f),q_2(f)\}=\{1,0\}$. Measuring the output of the CM we are able to deduce that the input of the XOR must correspond to either $I_1$ or to $I_4$ (or a certain linear combination of these two), since the related output is exactly $O_1$. Analogously, if for some signal $g\in\Hil$ our CM produces the output $\{q_1(g),q_2(g)\}=\{0,1\}$, we can deduce that $g$ must be a linear combination of $I_2$ or $I_3$, while it is surely neither $I_1$ nor $I_4$. Suppose now that, for some non trivial signal $h\in\Hil$, we find  $\{q_1(h),q_2(h)\}=\{0,0\}$. This means that $h$ is orthogonal to $e_j$, $j=1,2,3,4$. Hence $\dim(\Hil)$ must necessarily be larger than four, and $\E$ cannot be a basis for $\Hil$. Then we can introduce a fifth vector, $e_5:=\frac{1}{\|h\|}\,h$, which is surely orthogonal to all the vectors in $\E$, and a related projection operator $Q_3=P_5:=|e_5\left>\right<e_5|$, defined as the ones above. Of course, this can not be the case for the XOR we are considering here, but it might be the case in other situations. In this case, we {\em enlarge} the set $\E$ by adding $e_5$: $\E_1=\{e_j,\,j=1,2,3,4,5\}$, and we can now hope that $\E_1$ is an o.n. basis for $\Hil$. This fact, however, is again not guaranteed in general. Let us now consider a signal $f$ which is a general, normalized unknown combination of the vectors in $\E$. In this case, a natural way to relate $f$ with some of the possible inputs, is to compare $q_1(f)$ with $q_2(f)$, and to interpret these values as probabilities, as already proposed: the closer $q_1(f)$ to one, for instance, the higher the probability that $f$ is either $I_1$ or $I_4$. If, on the other hand, $q_2(f)\simeq1$, $f$ is most likely $I_2$ or $I_3$. Finally, if  $q_1(f)$ and $q_2(f)$ are of the same order, then we are not in a position to say much about $f$.

\vspace{2mm}

The OR gate works essentially in the same way. In this case the inputs are the same as before, but the correspondence with the outputs $O_1=0$ and $O_2=1$ is different. $I_1$ is associated to $O_1$, while all the other inputs are associated to $O_2$. Hence the relevant projection operators are $\tilde Q_1=P_1=diag(1,0,0,0)$ and $\tilde Q_2=P_2+P_3+P_4=diag(0,1,1,1)$. Therefore $\tilde q_1(f)=\|\tilde Q_1f\|^2=|f_1|^2$ and $\tilde q_2(f)=\|\tilde Q_2f\|^2=|f_2|^2+|f_3|^2+|f_4|^2$. Once again, the interpretation does not change.

\vspace{2mm}

We have further considered the possibility of having some (easy) map transforming the XOR  into the OR gate. This could have interesting consequences in concrete applications, of course. More explicitly, we have asked ourselves whether an invertible operator $U$ exists such that $\tilde Q_j=UQ_jU^{-1}$, $j=1,2$. The answer is negative, as one can easily understand. In fact, this map preserves traces, while it is clear that trace$(\tilde Q_j)\neq$ trace$(Q_j)$. A direct computation also shows that, even if $U$ is not necessarily invertible, the relations $\tilde Q_j=UQ_j$ and $\tilde Q_j=UQ_jU^\dagger$ cannot be true either, for any possible choice of $U$. So we could conclude that the projectors associated to the XOR and those associated to the OR gates cannot be linked by simple operations as those proposed so far. However, we can set up a different strategy which produces, up to a bijection, the desired result. The idea is simple: we associate, via an invertible map $\Phi$, the two $Q_j$ and the two $\tilde Q_j$ operators, to four orthonormal vectors $\varphi_j$, $\left<\varphi_j,\varphi_k\right>=\delta_{j,k}$: $\Phi(Q_j)=\varphi_j$,  and $\Phi(\tilde Q_j)=\varphi_{2+j}$, $j=1,2$. Now, let us introduce an operator $U=|\varphi_1\left>\right<\varphi_3|+|\varphi_2\left>\right<\varphi_4|$. The first remark is that $U^\dagger=|\varphi_3\left>\right<\varphi_1|+|\varphi_4\left>\right<\varphi_2|$. Notice also that $U$ is not invertible. Now, it is clear that $U$ makes the job. Indeed we have $U\varphi_3=\varphi_1$, $U\varphi_4=\varphi_2$, $U\varphi_2=U\varphi_1=0$, and $U^\dagger\varphi_3=U^\dagger\varphi_4=0$, $U^\dagger\varphi_1=\varphi_3$, $U^\dagger\varphi_2=\varphi_4$. Then, for instance,
$$
Q_1=\Phi^{-1}\left(U\,\Phi(\tilde Q_1)\right), \qquad Q_2=\Phi^{-1}\left(U\,\Phi(\tilde Q_2)\right).
$$
These equations show how to go from $(\tilde Q_1,\tilde Q_2)$ to $(Q_1,Q_2)$. The inverse transformation is implemented by $U^\dagger$.

\subsection{A second example: RGB colors}\label{sec:rgb}

Color clustering technique has broad applications in many  engineering, medical and computer science situations, see, for instance \cite{Fin,Li}.
We apply here the ideas introduced so far to the simple case in which we want associate an input color $f$ to some given reference colors.

We first consider three different reference points,  $R$ (which stands for {\em red}), $G$ (for {\em green}) and $B$,
(for {\em blue}). In the standard RGB-notation, they correspond to the following three o.n. vectors of
{$\Hil_{RGB}:={\Bbb R}^3$}: $R=(1,0,0)=e_1$, $G=(0,1,0)=e_2$ and $B=(0,0,1)=e_3$, which form a basis for $\Hil_{RGB}$.
They clearly produce three different orthogonal projection operators $Q_j=P_j=|e_j\left>\right<e_j|$, $j=1,2,3$,
and each signal $f\in\Hil_{RGB}$ (i.e. any other color) produces three different numbers $q_j(f)=|\left<e_j,f\right>|^2$.
Needless to say, if for instance $q_1(f)\gg \max\{q_2(f),q_3(f)\}$, then $f$ is {\em almost red}, while, if
$q_2(f)\gg \max\{q_1(f),q_3(f)\}$, then $f$ is {\em almost green}. As an example, let $f_R=(0.95,0.1,0.1)$. Then $q_1(f_R)=0.9025$,
while $q_2(f_R)=q_3(f_R)=0.01$. In this way, choosing a suitable $\epsilon>0$, we can construct three different clusters of signals,
each centered around a different reference point and of radius $\epsilon$. For instance, a cluster $K_\epsilon(R)$ centered in
$R$ contains all the signal $f$ such that $q_1(f)\in[\|f\|^2-\epsilon,\|f\|^2]$. Notice that the inequality $q_j(f)\leq\|f\|^2$ is automatically satisfied, for all $j$, because of the
Schwarz inequality. Notice also that, if $q_1(f)\geq\|f\|^2-\epsilon$, then, since $q_1(f)+q_2(f)+q_3(f)=\|f\|^2$,  $q_2(f)+q_3(f)=\|f\|^2-q_1(f)\leq \epsilon$.
Therefore $q_2(f),q_3(f)\notin[\|f\|^2-\epsilon,\|f\|^2]$ and, according to our previous interpretation, $f$ is {\em really different } from green and blue!
It could be worth observing that, since in RGB we cannot require $f$ to be normalized (otherwise we change the color!),
we must pay attention to the fact that having, for instance, $q_1(f)=1$ does not imply that $f\in K_\epsilon(R)$.
A simple counterexample is provided by the signal $f=(1,1,1)$. Of course we have $q_j(f)=1$ for all $j$. However, if we take $\epsilon$
reasonably small, $q_j(f)\geq\|f\|^2-\epsilon\geq3-\epsilon$ is false. Then $f$ does not belong to any (reasonably small) cluster centered in $R$, $G$ or $B$.

It could be also useful to consider our procedure from a different perspective. In fact, it is easy to understand
that if a signal $f$ is close to be  \textit{red}, then $\|f-R\|^2=q_1(f-R)+q_2(f-R)+q_3(f-R)\approx0$.
If $f_R$ is the signal previously defined,
we easily obtain $\|f_R-R\|^2=0.025$, while $\|f_R-G\|^2=\|f_R-B\|^2=1.63$. This allows to extend our CM to other reference points
which do not belong to $\E=\{e_j, j=1,2,3\}$. For instance, let us consider the following new reference points $P_1=(0.6,0,0.6)$ (which is purple) and $P_2=(0,0.8,0.2)$ (some sort of green).
Let us now consider the following inputs, which we need to classify, with respect to $P_1$ and $P_2$: $f_1=(0.8,0.1,1)$, $f_2=(0.3,0.6,0.1)$ and $f_3=(0.7,0.8,1)$.  We easily find that
$$
\|P_1-f_j\|^2\simeq\left\{
\begin{array}{ll}
0.21,\quad j=1,  \\
0.49,\quad j=2,  \\
0.81,\quad j=3,
\end{array}%
\right.
\qquad
\|P_2-f_j\|^2\simeq\left\{
\begin{array}{ll}
1.76,\quad j=1,  \\
0.13,\quad j=2,  \\
1.12,\quad j=3.
\end{array}%
\right.
$$
These results suggest that $f_1$ is closer to $P_1$ than the other inputs, and that $f_2$ is not very different from $P_2$, while $f_3$ is really another color, neither purple nor green.
 This is in fact what one observes, since $f_1$ is a different purple, $f_2$ is a dark green, while $f_3$ is a pale blue.To these same conclusions we arrive considering, as in formula (\ref{funF}), the function $F[f,g]$. We find
$$
F[P_1,f_j]\simeq\left\{
\begin{array}{ll}
0.01,\quad j=1,  \\
0.34,\quad j=2,  \\
0.22,\quad j=3,
\end{array}%
\right.
\qquad
F[P_2,f_j]\simeq\left\{
\begin{array}{ll}
0.78,\quad j=1,  \\
0.06,\quad j=2,  \\
0.36,\quad j=3.
\end{array}%
\right.
$$
These results suggest the same conclusions as above, but differences are made much more evident! Looking at these results we can safely say that $f_1$ belongs to a suitable cluster of $P_1$, while $f_2$ and $f_3$ do not (if $\epsilon$ is not large enough). Also, $f_2$ belongs to a suitable cluster of $P_2$, while $f_1$ and $f_3$ do not.

\subsection{Recognition of tones}\label{sec:tones}

Let us now consider the practical case in which our input signal is a sound sample which reproduces a single \textit{defined} tone
and we want to recognize this tone. The input signal is  a sound produced by any instrument (or by the human voice).
As we shall better explain later,  in some cases it is necessary to know which is the instrument used to play the sample.\\
Here we consider  input signals $\tilde f$ that are in {\em waveform audio file} format (i.e. WAV files)
with the common sampling frequency of  $44.1$ kHz. This roughly means that the
audio is recorded by sampling it $44.100$ times per second.
Therefore $\tilde f$ is a $N$ dimensional vector of $\Bbb R^{N}$, where $N$ depends on the time length $t_s$ of the sound in such a way that
$t_s=N/44100$. For simplicity we  consider here only sounds that are 1 second long, so that $N=44100$. Then
$\tilde f=(\tilde f_1,\tilde f_2,...\tilde f_{44100})\in \Bbb R^{44100}$.
However we do not directly work on the signal $\tilde f$. In fact, to deal with audio signals in { audio-processing engineering},
 one generally needs to evaluate
the Time Discrete Fourier Transform $\F_{D}[\tilde f](n)=f_n=\sum _{k=0}^{N-1}\tilde f_k \text{e}^{-ik2\pi n/N} $ of the original signal $\tilde f$,
and consider the resulting vector made by all modula of the Fourier modes $|\F_{D}[f](n)|$
(sometimes, in audio-processing engineering, also the square modulus is used, see \cite{Mar}).

Hence
our \textit{final} input signal is  $f=\sum_{n=1}^{22050}|f_n|e_{n}$, where $e_n$ is the $n$-th element of the canonical o.n basis
$\E$ of $\Hil_{sound}=\Bbb R^{22050}$
(as it is well known the algorithms computing the Fourier transform
of a real input signal of length $N$, returns a complex vector of $N/2$ elements, and this is why our space has dimension 22050).
We always consider our input signal $f$ normalized\footnote{We notice that, differently from the RGB case
in which the normalization
modifies the original color,
here the normalization of the signal does not affect the tone, but it only modifies the original loudness of the sound.} so that $\|f\|=1$. Here
 $\|.\|$ is the norm in $\Hil_{sound}$.\\
Each component of $f$ is related to a specific $frequency$ in the frequency domain. It is well known that if
the signal $f$  represents a specific tone $t$,
 then it has a \textit{fundamental frequency} $k_{ t}$ significantly excited ($|f_{k_t}|\gg 0$) and in general, depending on the kind of instrument
 which produces this tone, some of the frequencies
 multiplies of $k_{ t}$ are excited as well ($|f_{j\cdot k_t}|\gg 0$ for some $j=2,3,4,\ldots$).
 These latter frequencies are the so called \textit{harmonics}, which sometimes have even bigger amplitudes than the fundamental frequency.
 On the other hand, the neighboring frequencies of the fundamental and of the harmonics all decay rapidly to zero. A typical input signal $f_{A2}$  is shown in Fig.\ref{la},
 where $f_{A2}$ represent the tone A2 of an electric guitar playing the second open string (we consider here the standard guitar tuning EADGBE), and the fundamental
 frequency here is $110$Hz. The tone A2 means that we are playing the tone A in the second octave: each octave contains all the 12 semitones,
 i.e. C,C$^{\sharp}$,D,D$^{\sharp}$,E,F,F$^{\sharp}$,G,G$^{\sharp}$,A,A$^{\sharp}$,B. The difference between a tone on two different octaves is that the tone belonging
 to the higher octave plays sharper and its fundamental frequency is higher (for instance the tone A3 has the fundamental frequency 220Hz).
 We stress that the number of harmonics excited strictly depends on the instrument. For example we have noticed that a bass guitar
 can have  two or three harmonics significantly excited, while some digital piano can have only one significant harmonic.\\

\begin{figure}
\begin{center}\includegraphics[width=12cm]{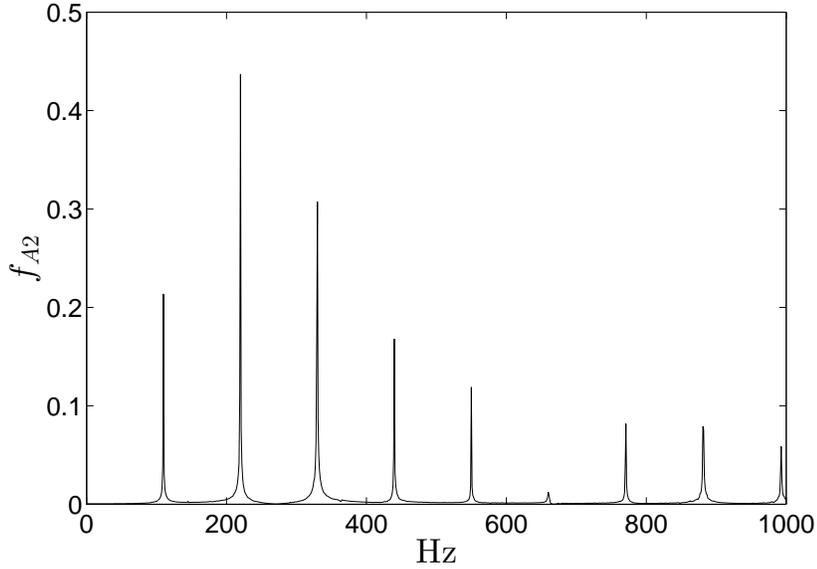}
\caption{The input signal representing the tone A2. The instrument is an electric guitar.
The fundamental frequency is 110Hz, and several harmonics (whose frequencies are  multiplies of the fundamental one) are significantly
excited.}
\label{la}
\end{center}
\end{figure}

As we did before for RGB, we need here to build a set $\mathcal{O}$
of reference inputs (tones)
to compare them with the unknown signal $f$.
A generic element $O^t$ in $\mathcal{O}$ can be  a vector in which the fundamental frequency  $k_{ t}$
and a finite number $n_h$ of its harmonics are excited, while
the other frequencies can be set to zero. In this way  $O^t$ is essentially a real input vector, and it represents a specific tone $t$ depending
on the value of $k_t$: for instance, the vector $O^{A2}$ representing the tone A2, has the components $O^{A2}_{110},O^{A2}_{220},...,O^{A2}_{n_h 110}$
significantly excited (and the other zero), the vector $O^{D3}$ has the components $O^{D3}_{147},O^{D3}_{294},...,O^{D3}_{n_h 147}$ significantly excited (and, again, the other zero), and so on.
For simplicity we assume that all the components excited of the various vectors $O^t$ are equal amongst themselves,  and that $O^t$ is normalized.
Hence, the set of all  reference tones is $\mathcal{O}_{n_h}=\{O^t: O^t_{j\cdot k_{t}}=\frac{1}{\sqrt{n_h+1}}\quad \text{and }
O^t_l=0, l\neq jk_{t}, j=1,...,n_h+1, \text{with } k_t\in\mathcal{F}_\mathcal{O} \}$. Here $\mathcal{F}_\mathcal{O}$ is the set of all the fundamental
frequencies $k_{t}$ that are determined
according the well known rule $k_{t}=[440\text{Hz}\cdot2^{(s/12)}]$, where $[\cdot]$ is the integer part function, 440Hz is the fundamental frequency of the tone A4, and $s$ is the number of
semitones between $t$ and A4.\\From now we shall focus only in signals representing tones in the octaves 2,3,4 and 5 (48 tones in total);
therefore $F_{\mathcal{O}}=\{65,69,73,...,932,988\}$ and a generic vector $O^t\in \mathcal{O}_{n_h}$
can be written as the following linear combination of the vectors $e_j$'s: $O^t=\frac{1}{\sqrt{n_h+1}}(e_{k_t}+e_{2k_t}+...e_{(n_h+1)k_t})$, $k_t\in F_{\mathcal{O}}$.
\\Once we have build the set of reference tones, we can apply our strategy to recognize the input signal $f$. Following the RGB example where the reference points were
colors different from R,G and B, we have computed all the values $\sum_{j=1}^{22050}q_{j}(f-O^t)=\|f-O^t\|^2$  and, since $\|f\|=\|O^t \|=1$, $F[f,O^t]=1-\left<f,O^t\right>$,
with $O^t\in \mathcal{O}_{n_h}$. Clearly, if $f$ plays a tone $\tilde t$ then we should have  $\|f-O^{\tilde t}\|^2<\|f-O^{t}\|^2$ and
$F[f,O^{\tilde t}]<F[f,O^t]$ for $\tilde t\neq t$.
However,  our ability to recognize the tone and the correct octave depends on the number $n_h$ of harmonics excited.
{  By taking $n_h=0$  we obviously obtain $\mathcal{O}_{0}=\{e_j,\,k_j\in \mathcal{F_\mathcal{O}}\}$,
and hence the references tones are the element of the basis of $H_{sound}$.}
But fixing $n_h=0$ we have to face with some drawbacks. In fact,
suppose to have a very simple input sound $f$ whose components follow the rule $f_k=(e^{-((k-110)/0.1)^2}+2e^{-((k-220)/0.1)^2}+e^{-((k-440)/0.1)^2})/\sqrt{6}$.
This signal represents the tone A2, as it has the fundamental frequency $k_f=110$ and two harmonics excited.
A straightforward computation shows that $F[f,e_{220}]<F[f,e_{j}]$ and $\|f-e_{220}\|^2<\|f-e_{j}\|^2$ for all $j\neq 220$; this means that the signal $f$ is wrongly recognized as a
tone A3. This issue
arises if we apply our strategy also to a real sound.
In Tables \ref{latabellan}-\ref{latabellaps} the results related to an input signal $f_{A2}$ representing the tone A2 (see Fig.\ref{la})
of an electric guitar are shown.
\beano
\begin{minipage}[t]{0.5\textwidth}
\begin{tabular}{l|c|c|c|c}
\toprule
$O^t$ &  2nd &  3rd &  4th &  5th\\
\midrule
C&1.999 & 1.993 & 1.997 & 1.998 \\
 C$\sharp$&1.998 & 1.992 & 1.995 & 1.970 \\
 D&1.998 & 1.992 & 1.998 & 1.994 \\
 D$\sharp$&1.997 & 1.994 & 1.987 & 2.000 \\
 E&1.998 & 1.995 & 1.127 & 1.984 \\
 F&1.993 & 1.989 & 1.985 & 1.995 \\
 F$\sharp$&1.996 & 1.988 & 1.994 & 1.996 \\
 G&1.984 & 1.981 & 1.999 & 1.999 \\
 G$\sharp$&1.977 & 1.962 & 1.994 & 1.995 \\
 A&1.491 & 0.8482 & 1.524 & 1.922 \\
 A$\sharp$&1.976 & 1.978 & 1.991 & 1.999 \\
 B&1.989 & 1.993 & 1.998 & 1.975 \\
 \bottomrule
\end{tabular}
\captionof{table}{$\|f_{A2}-O^t \|^2$ , $n_h=0$.}
\label{latabellan}
\end{minipage}
\begin{minipage}[t]{0.45\textwidth}
\hspace*{0.5cm}\begin{tabular}{l|c|c|c|c}
\toprule
  $O^t$&  2nd &  3rd &  4th &  5th\\
\midrule
C & 0.9997&0.9967&0.9987&0.9991\\
 C$\sharp$&0.9988&0.9961&0.9975&0.9849\\
 D&0.9992&0.9960&0.9990&0.9970\\
 D$\sharp$&0.9984&0.9969&0.9934&0.9998\\
E&0.9992&0.9974&0.5630&0.9919\\
 F&0.9967&0.9947&0.9924&0.9978\\
 F$\sharp$&0.9980&0.9938&0.9970&0.9980\\
 G&0.9919&0.9904&0.9995&0.9994\\
 G$\sharp$&0.9887&0.9811&0.9971&0.9974\\
 A&0.7460&0.4240&0.7620&0.9608\\
 A$\sharp$&0.9881&0.9892&0.9955&0.9995\\
 B&0.9946&0.9964&0.9989&0.9873\\
 \bottomrule
\end{tabular}
\captionof{table}{$F[f_{A2},O^{t}]$ , $n_h=0$.}
\label{latabellaps}
\end{minipage}
\enano

It is evident how, once again, the procedure wrongly recognizes the signal as a tone A3 both with the square norm value and with the function $F$,
and this is somewhat obvious as the first harmonic of $f$ (220Hz) has amplitude
greater than the amplitude of the fundamental (110Hz). Actually, with $n_h=0$, we have built a set or reference tones which are not really \textit{similar}
to an input signal given from the electric guitar.

However,
it is worth noting that the case $n_h=0$ well recognizes  the octave for input signals having
only the fundamental frequency significantly excited (for instance sounds recorded from some digital piano).
Generally, to avoid the mis-recognition of the octave we should construct the set $\mathcal{O}_{n_h}$  so that the reference tones $O^t$
are closer to the kind of signal we are analyzing. As previously said the signals relative to a guitar/bass guitar have 2 or more harmonics excited.
Therefore, it is more appropriate to take $n_h=2$. Hence the vectors in $\mathcal{O}_{2}$ are no more vectors of the basis $\E$ of $\Bbb R^{22050}$.
The results for $n_h=2$ are shown in Tables \ref{la23tabellan}-\ref{la23tabellaps}: in this case we can well recognize
both the tone and the octave. We have also checked (these results not shown here) if the recognition works for other input signals from an electric guitar/bass guitar,
and we were always able to well recognize the correct tone when $n_h=2$.

We have also applied our procedure to other instruments (violin), and in that case we were able to  recognize all
the input signals by taking $n_h=4$ (this is due to the fact that a violin has more harmonics significantly excited than an electric guitar).
\beano
\begin{minipage}[t]{0.45\textwidth}
\begin{tabular}{l|c|c|c|c}
\toprule
 $O^t$&  2nd &  3rd &  4th &  5th\\
\midrule
 C&1.984 & 1.992 & 1.992 & 1.991 \\
 C$\sharp$&1.984 & 1.985 & 1.993 & 1.953 \\
 D&1.988 & 1.883 & 1.930 & 1.957 \\
 D$\sharp$&1.980 & 1.980 & 1.989 & 1.957 \\
 E&1.991 & 1.699 & 1.577 & 1.891 \\
 F&1.984 & 1.990 & 1.991 & 1.984 \\
 F$\sharp$&1.986 & 1.987 & 1.982 & 1.989 \\
 G&1.985 & 1.993 & 1.988 & 1.987 \\
 G$\sharp$&1.971 & 1.992 & 1.993 & 1.976 \\
 A&1.152 & 1.479 & 1.836 & 1.938 \\
 A$\sharp$&1.970 & 1.992 & 1.994 & 1.960 \\
 B&1.985 & 1.988 & 1.848 & 1.817 \\
 \bottomrule
\end{tabular}
\captionof{table}{$\|f_{A2}-O^t \|^2$ , $n_h=2$}
\label{la23tabellan}
\end{minipage}
\begin{minipage}[t]{0.5\textwidth}
\begin{tabular}{l|c|c|c|c}
\toprule
 $O^t$&  2nd &  3rd &  4th &  5th\\
\midrule
C &0.9933&0.9972&0.9990	&0.9977\\
 C$\sharp$&	0.9935	&0.9961	&0.9974&	0.9997\\
 D&	0.9974&	0.9433	&0.9670&	0.9797\\
 D$\sharp$&	0.9970	&0.9915&	0.9955&	0.9843\\
E &	0.9981	&0.8540&	0.8870&	0.9486\\
 F&	0.9949&	0.9984&	0.9967&	0.9930\\
 F$\sharp$&0.9991&0.9957&0.9952 & 0.9988\\
 G&	0.9959&0.9982&	0.9959&	0.9969\\
 G$\sharp$&	0.9949&	0.9975&	0.9966&	0.9880\\
 A&	0.7140&	0.8480	&0.9517&	0.9797\\
 A$\sharp$&	0.9873&	0.9971&	0.9979&	0.9805\\
 B&	0.9950&	0.9980	&0.9246&	0.9186\\
 \bottomrule
\end{tabular}
\captionof{table}{$F[f_{A2},O^{t}]$ , $n_h=2$.}
\label{la23tabellaps}
\end{minipage}
\enano

\subsection{Celiac disease diagnosis}\label{sec:celiac}

We now consider an application of a completely different kind, i.e. the case in which our inputs are vectors containing numerical values representing symptoms, signs and laboratory findings useful to diagnose the celiac disease. Our strategy is based on the idea that we can use a  large set of patients with known diagnosis (celiac or not celiac), to determine if a new patient is celiac or not.
To borrow a common word used in clustering procedures, we use a $training$ dataset $\mathcal S_{train}$ made of 300 input data $I_j$, $j=1,2,\ldots,300$, for which the corresponding outputs (the diagnosis) are $O_0$ (no celiac) or $O_1$ (celiac). Each $I_j$ is charachterized by 16 symptoms  (\textit{Abdominal distention,Cephalea / Migraine,	 Chronic diarrhoea,	Dyspepsia,	 Epigastric heartburn,	Face' swelling,	Fatigue / Astenia,	Growth failure,	 Hair loss,	Nausea,	 Recurrent abdominal pain,	Recurrent miscarriage,	Regurgitation,	Steatorrhoea	 Vomiting,	 Weight loss}), 11 signs (\textit{Abdominal meteorism,	 Abdominal pain,	Alopecia,	 Amenorrhea,	 Bloating,	 Dermatitis herpetiformis,	Mouth ulcer,	Nail dystrophy,	Paleness, 	 Shortness,	 Thinness}) and 8 laboratory findings (\textit{Anemia,	Ab anti Tg positivity,	Ab anti TPO positivity,	 High TSH levels,	Hypertransaminasemia,	Hypocalcemia,	Hypoferritinemia,	Low serum iron}). Each value is an integer number ranging between 0 and 10 (intermediate value are possible), where 0 (resp. 10) can be interpreted as absence (resp. maximal presence) of symptoms/signs or  very low (resp. very high) values of laboratory findings.  It is obvious that quantifying symptoms or signs is not an easy task, and values should be validated by an expert physician (values contained in our dataset are obtained after accurate	medical examinations from physicians of the University of Palermo).  Following what done in the previous section, our Hilbert space is $\Hil=\Bbb R^{35}$, and each patient is seen as a vector $f=(f_1,...,f_{35})\in \Bbb R^{35}$ to which an output $O_j$, $j=0,1$ is associated; the o.n. basis is the usual one, $e_j$, $j=1,2,\ldots,35$. For simplicity we write $\mathcal S_{train}=\mathcal S_0 \bigcup \mathcal S_1$, with $S_0\bigcap S_1=\emptyset$, being $S_0$ the set of inputs having $O_0$ as output, and $S_1$ those having $O_1$.

To check if a new patient $f_{new}$ is celiac or not, we compute $d_0=\min_{f_0\in{S_0}} F[f_{new},f_0]$ and  $d_1=\min_{f_1\in{S_1}} F[f_{new},f_1]$, being $F$ defined in \eqref{funF}. If $d_0>d_1$ we mark the patient as not celiac, while he is marked celiac whether $d_0<d_1$; this is somehow obvious to understand, because if $d_0>d_1$ (resp. $d_0<d_1$) means that the  patient $f_{new}$ has similarities with a patient that is already marked as not celiac (resp. celiac). The case $d_0=d_1$, which however is quite  unlikely and was never observed in our computations,  represents a doubtful situation: this means that we have found two vectors $f_0 \in \mathcal S_0,f_1 \in \mathcal S_0$ so that $F[f_{new},f_0]=F[f_{new},f_1]$. In this case we could remove  $f_0,f_1$ from $\mathcal S_0$ and $\mathcal S_1$ and compute the new values of $d_0$ and $d_1$ to deduce our diagnosis; notice however that it is still possible, thought quite unluckily, to get again $d_0=d_1$. In this case, we have to repeat once more this procedure.

To test this idea we have considered a set $S_{new}$ containing 30 new patients, for which  the diagnosis is known (otherwise it is impossible to check if the procedure works).
By applying our procedure we have obtained that 26 over 30 diagnoses were correct. In particular the 4 wrong diagnosis contain 1 patient erroneously marked as not celiac , and 3 patients erroneously marked as celiac. It is worth noting that the same results, with the same wrong diagnoses, are obtained if we determine $d_0$ and $d_1$ as $d_0=\min_{f_0 \in \mathcal S_0}\|f_{new}-f_0\|^2$ and $d_1=\min_{f_1 \in \mathcal S_1}\|f_{new}-f_1\|^2$.  To compare these results with other classical method we have build a  decision tree based on the C4.5 algorithm, \cite{Qui93}, which uses the concept of information entropy. The training dataset used is $\mathcal S_{train}$ and calculation are performed through $Weka$ software, \cite{Weka}. The  application of the decision tree
has determined 7 wrong diagnoses: what is relevant is that 4 of these erroneous diagnoses are the same of our procedure, but, worst than ours, the decision tree makes three more mistakes. Even if the test set is at the moment not too large, the results of our analysis already show that the efficiency of our method is at least comparable with other classical techniques.

We are willing to apply our strategy also to the diagnosis of the Kawasaki syndrome, but we are still completing the creation of a dataset sufficiently rich to train the system. Of course, still a different possibility would consists in comparing our results with those given by some artificial neural network (ANN). Indeed, for the celiac disease, this has been done, and our results suggest again that our CM works better than the ANN.

\section{The role of the frames}\label{sec:frames}

In this section we replace o.n. bases of $\Hil$ with frames. Frame theory has been successfully used in many pure and applied mathematical contexts.
This includes time-frequency analysis \cite{Gro}, image processing/reconstruction \cite{Chan}, quantum measurements
\cite{Eld}, sampling theory \cite{Eld2}, data recovering \cite{Cai} and bioimaging \cite{Kov}, to name a few.
As we will explicitly show, frames are useful in our context  since they allow our CM to work with noised signals and,
in particular, they allows to recognize which is the relevant cluster for a signal $f\in\Hil$ to which some noise $\nu$ has been added. This could be important when,
for instance, the signal is transmitted from a source to a receiver, or when some background noise is present,
or in many other concrete situations.  We stress that our transmission is subjected only to some noising effect, and we shall not deal in this case with data loss or corruption. However the latter cases are also typical settings in which frames are used.  For instance in \cite{GKK01}  a quantized frame expansion has been used to guarantee robustness to the transmission of packet network, while in \cite{CK03} the authors give a complete classification of frames with respect
to their robustness to erasures in the same setting (we also mention \cite{HP04} where the authors introduced a measure for optimality of frames under erasures).

We now describe our mathematical setting based on frames. Let $\F_\Psi=\{\Psi_j, \,j=1,\ldots,N\}$ be an $(A,B)$-frame for the finite dimensional Hilbert space $\Hil$, see \cite{ole,dau} and the Appendix, for few useful results. Then, calling $\F_{\tilde\Psi}=\{\tilde\Psi_j, \,j=1,\ldots,N\}$ its (canonical) dual frame, any vector $f$ in $\Hil$ can be written as
$$
f=\sum_{j=1}^N\left<\Psi_j,f\right>\tilde\Psi_j=\sum_{j=1}^N\left<\tilde\Psi_j,f\right>\Psi_j.
$$
Let $f,g\in\Hil$. We introduce the $dissimilarity$ $measures$ $\Delta$ and $\nabla$:
\be
\Delta(f,g)=\max\left\{\sup_j\left|\left<\Psi_j,f-g\right>\right|,\sup_j\left|\left<\tilde\Psi_j,f-g\right>\right|\right\},
\label{31}\en
and
\be
\nabla(f,g)=\min\left\{\sup_j\left|\left<\Psi_j,f-g\right>\right|,\sup_j\left|\left<\tilde\Psi_j,f-g\right>\right|\right\}.
\label{32}\en
It is clear that both $\Delta(f,g)$ and $\nabla(f,g)$ are non negative and that $\Delta(f,g)\geq\nabla(f,g)$ for all $f,g\in\Hil$. Moreover, if $f=g$, then $\Delta(f,g)=\nabla(f,g)=0$. It is possible to check that both $\Delta(f,g)$ and $\nabla(f,g)$ reduce essentially to the norm distance when $\F_\Psi$ is an o.n. basis. In fact, if this is the case, it is well known that $\tilde\Psi_j=\Psi_j$ for all $j$. Then $\Delta(f,g)=\nabla(f,g)=\sup_j\left|\left<\Psi_j,f-g\right>\right|$, and the following proposition holds:
\begin{prop}\label{prop1}
Let  $\F_\Psi$ be an o.n. basis. Then, taking $f,g\in\Hil$ with $\|f-g\|\leq\epsilon$, we have $\Delta(f,g)\leq\epsilon$. Viceversa, if $\Delta(f,g)\leq\epsilon$, then $\|f-g\|\leq\epsilon\,\sqrt{\dim(\Hil)}$
\end{prop}
\begin{proof}
The first result follows from the Schwarz inequality: $\left|\left<\Psi_j,f-g\right>\right|\leq \|\Psi_j\|\|f-g\|=\|f-g\|$. Then $\Delta(f,g)=\sup_j\left|\left<\Psi_j,f-g\right>\right|\leq \sup_j\|f-g\|=\|f-g\|\leq\epsilon$.

To prove the second statement, we use the Perseval equality as follows
$$
\|f-g\|^2=\sum_{j=1}^N\left|\left<\Psi_j,f-g\right>\right|^2\leq \epsilon^2 \sum_{j=1}^N 1=\epsilon^2\,\dim(\Hil).
$$
Here we have used the fact that, since by assumption $\Delta(f,g)\leq\epsilon$,  $\sup_j\left|\left<\Psi_j,f-g\right>\right|\leq\epsilon$, so that  $\left|\left<\Psi_j,f-g\right>\right|\leq\epsilon$ for all $j$.

\end{proof}

Of course, since in this proposition $\F_\Psi$ is assumed to be an o.n. basis for $\Hil$, $\dim(\Hil)=N$. Finally, this result also proves that both $\Delta(f,g)$ and $\nabla(f,g)$ can be seen as a sort of {\em extended distances}, not very different from the distance defined by the norm in $\Hil$. This conclusion is also true when $\F_\Psi$ is a frame (and not necessarily an o.n. basis). In fact, in this case,  we have (i) $\Delta(f,g)=\Delta(g,f)$; (ii) $\Delta(f,g)>0$, if $f\neq g$, and $\Delta(f,f)=0$; (iii) $\Delta(f,g)\leq \Delta(f,h)+\Delta(h,g)$, for all $f,g,h\in\Hil$. As for $\nabla(f,g)$, the analogous of (i) and (ii) are also satisfied, but the triangular inequality is not, in general. Hence $\nabla$ is not a distance, but still can be used, to some extent, to measure the difference between two vectors $f$ and $g$. In this case, if $\Delta(f,g)\leq\epsilon$ and if $\F_\Psi$ is a frame, not necessarily an o.n. basis, we get \be\|f-g\|\leq\epsilon\min\left(\sum_{j=1}^N\|\Psi_j\|,\sum_{j=1}^N\|\tilde\Psi_j\|\right).\label{33}\en
Again, if $\Delta(f,g)\leq\epsilon$, and if both $\sum_{j=1}^N\|\Psi_j\|$ and $\sum_{j=1}^N\|\tilde\Psi_j\|$ are finite and independent of $\epsilon$, $\|f-g\|$ goes to zero with $\epsilon$.

\vspace{2mm}

{\bf Remark:--} if we again assume for a moment $\F_\Psi$ to be an o.n. basis for $\Hil$, then $\sum_{j=1}^N\|\Psi_j\|=\sum_{j=1}^N\|\tilde\Psi_j\|=N$. Then inequality (\ref{33}) is weaker than the one deduced in Proposition \ref{prop1}, since, whenever $\dim(\Hil)>1$, $\sqrt{\dim(\Hil)}<\dim(\Hil)$.
\vspace{3mm}

We are now ready to use this framework for our original clustering problem. For that, let now $P\in\Hil$ be our {\em reference point}, $f\in\Hil$ and $f_{noised}=f+\nu$ the signal we want to attach to $P$, with the noise added. We start defining three, in principle, different neighborings of $P$:
\be
\begin{array}{ll} K_\epsilon^{\|.\|}(P):=\{f\in\Hil\,: \|P-f\|\leq\epsilon\},\\
K_\epsilon^{\Delta}(P):=\{f\in\Hil\,: \Delta(f,P)\leq\epsilon\},\\
K_\epsilon^{\nabla}(P):=\{f\in\Hil\,: \nabla(f,P)\leq\epsilon\}.\\
\end{array}
\label{34}\en
Of course, from what we have seen, there exist connections between these sets. For instance, since $\Delta(f,P)\geq\nabla(f,P)$ for all $f\in\Hil$, we deduce that $K_\epsilon^{\Delta}(P)\subseteq K_\epsilon^{\nabla}(P)$. Also, calling $M=\max\{\sum_{j=1}^N\|\Psi_j\|,\sum_{j=1}^N\|\tilde\Psi_j\|\}$, inequality (\ref{33}) shows that, if $f\in K_\epsilon^{\Delta}(P)$, then $f\in K_{M\epsilon}^{\|.\|}(P)$. Other properties of this kind could also be deduced.

We will now see why frames could be more useful than o.n. bases in a clustering procedure involving noised signal.

\subsection{An example in $\Hil={\Bbb R}^3$}\label{sec:exahil3}

We begin with a very simple example, living in a three-dimensional space. Let $P=(1,2,3)$ be a point (or a vector) in $\Hil$. Our aim is to construct, for this reference point, the clusters in (\ref{34}) and we want to show, in particular, that using $K_\epsilon^{\|.\|}(P)$ could be not a proper choice to recognize signals affected by noise.

In fact, it is very easy to construct an example: let us consider the following {\em clean} signal, $f=(1.1,2,3)$, and a small noise $\nu=(0,0.1,0)$. Hence $f_{noised}=(1.1,2.1,3)$. Now,  $f-P=(0.1,0,0)$ while $f_{noised}-P=(0.1,0.1,0)$. Let us now further fix $\epsilon=0.1$ as the size of the clusters. Then the original signal, $f$, belongs to $K_\epsilon^{\|.\|}(P)$, while the noised signal does not. In fact we have $\|f-P\|=0.1\leq\epsilon$, while $\|f_{noised}-P\|=\sqrt{0.02}\geq\epsilon$.

This means that, when we send $f$ to a receiver $\cal R$, since along the way the original signal is noised and $f_{noised}$ is what it is received by $\cal R$, the receiver could determine that $f$  does not belong to $K_\epsilon^{\|.\|}(P)$. This is a wrong conclusion, since $f\in K_\epsilon^{\|.\|}(P)$. Let us now show that this can be avoided using frames rather than o.n. bases\footnote{Of course, using o.n. bases is equivalent to use norms, because of the Parseval equality. Hence $K_\epsilon^{\|.\|}(P)$ could also be defined in terms of some o.n. basis of $\Hil$.}.

\vspace{2mm}

Let us now consider the tight dual frame $\F_\Psi$ introduced in the Appendix, and its dual frame. A straightforward computation shows first that $\Delta(f,P)=\max\{0.1,0.08\}=0.1\leq\epsilon$ while $\nabla(f,P)=\min\{0.1,0.08\}=0.08\leq\epsilon$. Hence $f\in K_\epsilon^{\nabla}(P)\cap K_\epsilon^{\Delta}(P)$. So $f$ is {\em close} to $P$ in any of the possible ways we have considered here. Now, let us see what happens for $f_{noised}$. It is again very simple to check that $\Delta(f_{noised},P)=\max\{0.1,0.08\}=0.1\leq\epsilon$ and that $\nabla(f_{noised},P)=\min\{0.1,0.08\}=0.08\leq\epsilon$. Therefore, also the noised signal belong to the same cluster as $f$, whichever choice we make. Then, in this case, both $\Delta$ and $\nabla$ work fine for our purposes.

\vspace{2mm}

{\bf Remark:--} Of course, it is not difficult to adapt this example to RGB, since the Hilbert space is exactly the same, so that we can use the same frame.
However, we will not do it here, since this explicit application is now absolutely straightforward. Rather than that, we prefer to focus now on recognition of noised tones, for which the Hilbert space is significantly bigger.

\subsection{Recognition of tones, part 2}\label{sec:rectones2}
In this section we shall see how frames can be useful in the recognition of a tone in which some noise is added. Let us consider for instance the situation shown
in Fig.\ref{la2noise}. A noised signal $g_{A2n}$ is generated by adding some noise $\nu$ to the first 1000 components of an input signal $g_{A2}$
representing some kind of digital sound
reproducing the tone $A2$. Therefore, $g_{A2n}=g_{A2}+\nu$, where $\nu$ is a vector of $\Bbb R^{22050}$ whose first 1000 components are random number ranging from 0 to 0.1, while the others are
set to 0.
\begin{figure}
\begin{center}\includegraphics[width=14cm]{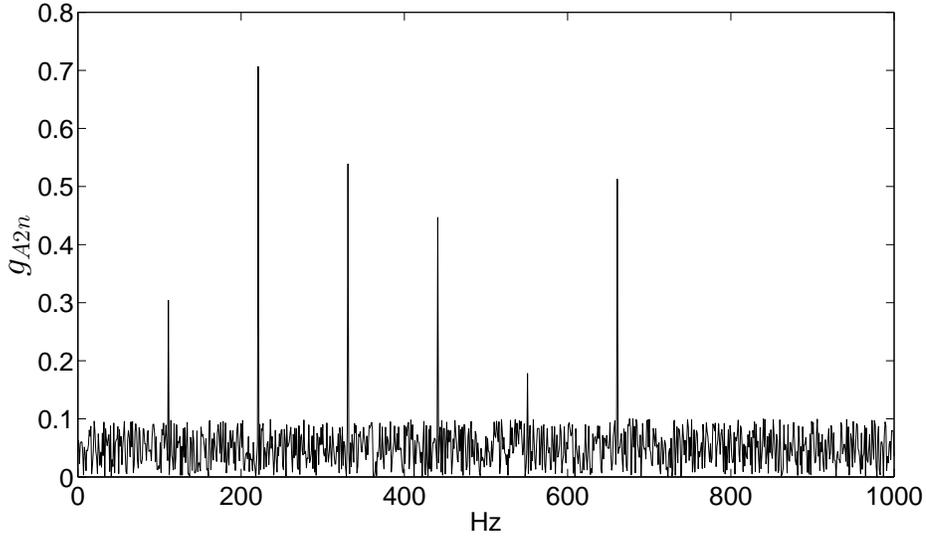}
\caption{The input signal $g_{A2n}$ is obtained by adding some noise to the first 1000 components of an input signal reproducing the tone A2.}
\label{la2noise}
\end{center}
\end{figure}
Following what was proposed in Section II.3, to recognize the tone  $g_{A2n}$ { we first build the set of reference tones $O_2$ ($n_h=2$), and then we
can evaluate the various squared norms
$\sum_{j=1}^{22050}q_{j}(g_{A2n}-O^{t})=\|g_{A2n}-O^{t} \|^2$,   for all the reference tones $O^t\in O_2$. As $g_{A2n}$ is a noised $A2$ tone,
we should expect that $\|g_{A2n}-O^{A2} \|^2<\|g_{A2n}-O^{j} \|^2$, for all $O^j\neq O^{A2}$},
but, on the contrary, the tone is not correctly recognized, as we have obtained
that the minimum value is $\|g_{A2n}-O^{A3} \|^2=3.792$, while $\|g_{A2n}-O^{A2} \|^2=3.927$.
Therefore the noised signal $g_{A2n}$ is wrongly recognized as a tone $A3$, and this is clearly due to the presence of the noise in the signal.

On the other hand, we can define $$\F_\Psi=\left\{\Psi_1=e_1, \Psi_2=\frac{1}{2}\,e_1,\Psi_3=e_2, \Psi_4=\frac{1}{2}\,e_2,...,\Psi_{44099}=e_{22050},
\Psi_{44100}=\frac{1}{2}\,e_{22050}\right\}$$
being as usually $\E=\{e_1,e_2,...,e_{22050}\}$  the canonical o.n. basis of $\Hil_{sound}$. It is easy to check that for each $f\in\Hil_{sound}$,
 $\sum_{j=1}^{44100}|<\Psi_j,f>|^2=\frac{5}{4}\|f\|^2$, which implies that $\F_\Psi$ is a tight frame with $A=\frac{5}{4}$.
Its dual frame is $\F_{\tilde\Psi}=\left\{\tilde\Psi_j=\frac{4}{5}\Psi_j,\,j=1,\ldots,44100\right\}$,
and indeed it is also  easy to check formula (\ref{a3}).
It is clear that if we want to recognize an input signal $f$ we have to check among all the possible values $\Delta(f-O^{t})$ and $\nabla(f-O^{t})$
and look for the minima. By an easy computation we have obtained that the minimum values are $\Delta(g_{A2n}-O^{A2})=0.5126$ and $\nabla(g_{A2n}-O^{A2})=0.41$,
while $\Delta(g_{A2n}-O^{A3})=0.539$ and $\nabla(g_{A2n}-O^{A3})=0.431$. This means that
the noised signal $g_{A2n}$ is correctly recognized as a tone $A2$  adopting both $\Delta$ or $\nabla$.

 Frames are therefore useful in this application as they can {overcome} the issues related to the presence of some noise to some given signal. However, it should also be observed that using the difference given by the norms sometimes work efficiently enough: in other words, most of the times the noised signals are recognized both using the norm of the difference, as in Section II, or using $\nabla$ or $\Delta$. Still sometimes, while $\nabla$ or $\Delta$ recognize the noised signal, the norm does not. This, in our opinion, makes frames important for this kind of applications.

\section{Conclusions}\label{sec:conclusion}

In this paper we have used some ideas, very common in quantum mechanics, functional analysis and harmonic analysis, to construct what we have called {\em clustering machines}, i.e. strategies useful to classify some inputs dividing them in different clusters. Also, we have shown how frames can be used when the signals are affected by some noise, as when the signal is transferred from one place to another. We believe that frames is a more natural choice, with respect to o.n. bases, exactly because of their redundancies. This is not a surprise, since overcomplete sets are quite often used exactly for taking into account a possible loss of information.

We have applied successfully our strategy to logic gates, RGB colors, and recognition of tones, both for un-noised and for noised signals, and to the  diagnosis of the celiac disease for which we have used a set of symptoms and signs as an un-noised signal, and we have compared our conclusions with those deduced adopting other clustering strategies. Concerning the next steps, we plan to compare our CMs also with other methods coming, for instance, from fuzzy logic.


\vspace{8mm}

\renewcommand{\theequation}{\Alph{section}.\arabic{equation}}

 \section*{Appendix: few facts on frames}\label{sec:appendix}

Let $A$ and $B$ be such that $0<A\leq B<\infty$, and let $\J\subseteq{\Bbb N}$.

\begin{defn} A set of vectors
$\F_\Psi=\{\Psi_j\in\Hil,\,j\in\J\}$ is an  $(A,B)-$frame if,  $\forall\,f\in\Hil$, \be A\|f\|^2\leq
\sum_{j\in\J}|<\Psi_j,f>|^2\leq B\|f\|^2.\label{a1}\en If $A=B$, the frame is called  tight.
\end{defn}

It is well known that, while all o.n. bases are tight frames with $A=1$, the converse is true only when $\|\Psi_j\|=1$ for all $j\in\J$. The theory of frames is extremely rich and elegant. The only aspect we will discuss here, which is the one really useful for us, is how a resolution of the identity can be recovered out of $\F_\Psi$. For that, the first step consists in introducing the so called {\em frame operator} $F: \Hil\rightarrow \lmin^2(\J)$ defined as follows
 \be \forall f\in\Hil \quad\Rightarrow\quad (F f)_j:=<\Psi_j,f>, \quad \forall
j\in\J.\label{a2}\en
Its adjoint turns out to be $F^\dagger c=\sum_{j\in\J}c_j\Psi_j$, for all $c=\{c_j,\,j\in\J\}\in \lmin^2(\J)$, and $F^\dagger: \lmin^2(\J)\rightarrow \Hil$. Now, since for all $f\in\Hil$,
$$\sum_{j\in\J}|<\Psi_j,f>|^2=<f,F^\dagger Ff>,$$
equation \eqref{a1} can be rewritten, more compactly, as $A\1\leq F^\dagger F\leq B\1$. This double inequality implies, in particular, that $\F_1:=F^\dagger F$ is invertible, and that its inverse is also bounded. Let us now define the new vectors $\tilde \Psi_j=\F_1^{-1} \Psi_j$, $j\in\J$, and the new {\em dual} frame $\F_{\tilde\Psi}=\{\tilde\Psi_j\in\Hil,\,j\in\J\}$. Then, calling $\tilde F=F\F_1^{-1}$, we can check that $\tilde F^\dagger F=F^\dagger \tilde F=\1$. As a consequence of these equalities, we find that
\be
f=\sum_{j\in\J}\left<\Psi_j,f\right>\tilde\Psi_j=\sum_{j\in\J}\left<\tilde\Psi_j,f\right>\Psi_j,
\label{a3}\en
which can be written, in bra-ket language, as $\sum_{j\in\J}|\Psi_j><\tilde\Psi_j|=\sum_{j\in\J}|\tilde\Psi_j><\Psi_j|=\1$. Incidentally we recall that the dual frame of $\F_{\tilde\Psi}$, $\F_{\tilde{\tilde\Psi}}$, coincides with $\F_{\Psi}$ itself.

In \cite{dau} it is discussed in some details how to construct explicitly, out of $\F_\Psi$, the vectors $\tilde\Psi_j$. The technique is perturbative, and it works well when $A$ and $B$ are very close each other. Since in this paper we just consider tight frames, the construction is much easier. In fact, in this case  $\F_1= A\,\1$, and therefore $\F_1^{-1}= \frac{1}{A}\,\1$ and
$\tilde\Psi_j=\F_1^{-1}\Psi_j=\frac{1}{A}\,\Psi_j$.

\vspace{2mm}

{\bf Example:--} let $\E={e_1,e_2,e_3}$ be the canonical o.n. basis of $\Hil={\Bbb C}^3$, and let
$$\F_\Psi=\left\{\Psi_1=e_1, \Psi_2=\frac{1}{2}\,e_1,\Psi_3=e_2, \Psi_4=\frac{1}{2}\,e_2,\Psi_5=e_3, \Psi_6=\frac{1}{2}\,e_3\right\}.$$
Let $f\in\Hil$ be a generic vector.
Since $\sum_{j=1}^6|<\Psi_j,f>|^2=\frac{5}{4}\|f\|^2$, $\F$ is a tight frame with $A=\frac{5}{4}$.
Its dual frame is $\F_{\tilde\Psi}=\left\{\tilde\Psi_j=\frac{4}{5}\Psi_j,\,j=1,\ldots,6\right\}$,
and indeed it is easy to check formula \eqref{a3}.

\end{document}